\documentclass[12pt, a4paper]{article}
\usepackage{}

\usepackage{amsfonts}
\usepackage{bbm}
\usepackage{mathrsfs}
\usepackage{latexsym}
\usepackage{graphicx}
\usepackage{amssymb}
\usepackage{amsmath}
\usepackage{color,xcolor}
\usepackage{mathtools}
\usepackage{cite}
\usepackage{amsthm}
\newtheorem{theorem}{Theorem}[section]
\newtheorem{lemma}[theorem]{Lemma}
\newtheorem{corollary}[theorem]{Corollary}

\newtheorem{remark}[theorem]{Remark}

\title{{\Large \bf  The general Albertson irregularity index of graphs
\thanks{ Supported by the National Natural Science Foundation of China (No. 11771443).}~}}
\author{Zhen Lin$^{a}$\thanks{Corresponding author. E-mail addresses: lnlinzhen@163.com(Z. Lin), zhouting@cumt.edu.cn(T. Zhou), miaolianying@cumt.edu.cn (L. Miao).},  Ting Zhou$^b$, Lianying Miao$^b$\\
{\footnotesize $^a$School of Mathematics and Statistics,
Qinghai Normal University,}\\ {\footnotesize  Xining, 810008, Qinghai, P.R. China}\\
\footnotesize  $^b$School of Mathematics, China University of Mining and Technology,\\ \footnotesize Xuzhou, 221116, Jiangsu, P.R.
China
}

\date{}
\begin{document}
\openup 1.0\jot
\date{}\maketitle
\begin{abstract}
We introduce the general Albertson irregularity index of a connected graph $G$ and define it as $A_{p}(G) =(\sum_{uv\in E(G)}|d(u)-d(v)|^p)^{\frac{1}{p}}$, where $p$ is a positive real number and $d(v)$ is the degree of the vertex $v$ in $G$. The new index is not only generalization of the well-known Albertson irregularity index and $\sigma$-index, but also it is the Minkowski norm of the degree of vertex. We present lower and upper bounds on the general Albertson irregularity index. In addition, we study the extremal value on the general Albertson irregularity index for trees of given order. Finally, we give the calculation formula of the general Albertson index of generalized Bethe trees and Kragujevac trees.

\bigskip


\end{abstract}
\baselineskip 20pt

\section{\large Introduction}

Let $G$ be a simple undirected connected graph with vertex set $V(G)$ and edge set $E(G)$. For $v\in V(G)$, $N(v)$ denotes the set of all neighbors of $v$, and $d(v)=|N(v)|$ denotes the degree of vertex $v$ in $G$. The minimum and the maximum degree of $G$ are denoted by $\delta(G)$ and $\Delta(G)$, or simply $\delta$ and $\Delta$, respectively. A pendant vertex of $G$ is a vertex of degree one. A graph $G$ is called $(\Delta, \delta)$-semiregular if $\{d(u), d(v)\} = \{\Delta, \delta\}$ holds for all edges $uv\in E(G)$. Denote by $P_n$ and $K_{1,\,n-1}$ the path and star with $n$ vertices, respectively.

In 1997, the Albertson irregularity index of a connected graph $G$, introduced by Albertson \cite{A}, is defined as
$$Alb(G) =\sum_{uv\in E(G)}|d(u)-d(v)|.$$
This index has been of interest to mathematicians, chemists and scientists from related fields due to the fact that the Albertson irregularity index plays a major role in irregularity measures of graphs \cite{ACD, AD, CHL, DR, HR}, predicting the biological activities and properties of chemical compounds in the QSAR/QSPR modeling \cite{GHM, RSDH} and the quantitative characterization of network heterogeneity \cite{E}. By the natural extension of the Albertson irregularity index, Gutman et al. \cite{GTYCC} recently proposed the $\sigma$-index as follows:
$$\sigma(G)=\sum_{uv\in E(G)}(d(u)-d(v))^2=F-2M_2,$$
where $F$ and $M_2$ are well-known the forgotten topological index and the second Zagreb index of a graph $G$, respectively. Recently, the $\sigma$-index of a connected graph $G$ is studied, such as the characterization of extremal graphs \cite{ADG} and mathematical relations between the $\sigma$-index and other graph irregularity indices \cite{R1}.

The generalization of topological index is a trend of mathematical chemistry in recent years. Many classical topological indices are generalized, such as the general Randi\'{c} index \cite{BE}, the first general Zagreb index \cite{LZ}, the general sum-connectivity index \cite{ZT}, the general eccentric connectivity index \cite{VM}, etc. Motivated by this fact, we propose the general Albertson irregularity index of a graph $G$ as follows:
$$A_{p}(G) =\left(\sum\limits_{uv\in E(G)}|d(u)-d(v)|^p\right)^{\frac{1}{p}},$$
where $ p$ is a positive real number. Evidently, $A_1(G)=Alb(G)$ and $A_2^2(G)=\sigma(G)$. The other motivation is that the topological index formed from distance function of the degree of vertex has attracted extensive attention of scholars. In 2021, Gutman \cite{G} proposed the Sombor index of a graph $G$ and defined it as $SO(G) =\sum_{uv\in E(G)}\sqrt{d^2(u)+d^2(v)}$, which is the Euclidean norm of $d(u)$ and $d(v)$. According to Gutman \cite{G1}, it is imaginable to use other distance function to study properties of graphs. Based on this, it is not difficult to find that $A_{p}(G)$ is the Minkowski norm of $d(u)$ and $d(v)$, which is unification of absolute distance, Euclidean distance and Chebyshev distance. Hence $A_{p}(G)=\Delta-\delta$ as $p$ becomes infinite. In particular, $A_{p}(G)$ is the $l_p$-norm of $d(u)$ and $d(v)$ for $p\geq 1$.

We will first recall some useful notions and lemmas used further in Section 2. In Section 3, upper and lower bounds on the general Albertson irregularity index of graphs are given, and the extremal graphs are characterized. In Section 4, the first two trees with minimum general Albertson irregularity index are determined in all trees of fixed order. In Section 5, the general Albertson index of the well-known generalized Bethe trees and Kragujevac trees is obtained.

\section{\large  Preliminaries}

Let $u\vee G$ be the graph by adding all edges between the vertex $u$ and $V(G)$. The first general Zagreb index of a graph $G$ is defined as $Z_p(G)=\sum_{v\in V(G)}d^p(v)$ for any real number $p$. The distance between two vertices $u, v \in V(G)$, denoted by $d(u,v)$, is defined as the length of a shortest path between $u$ and $v$. The eccentricity of $v$, $\varepsilon(v)$, is the distance between $v$ and any vertex which is furthest from $v$ in $G$. The line graph $L(G)$ is the graph whose vertex set are the edges in $G$, where two vertices are adjacent if the corresponding edges in $G$ have a common vertex. Let $\mathscr{T}_{n}$ be the set of trees with $n$ vertices. A spider is a tree with at most one vertex of degree more than two.

\begin{lemma}{\bf (Power mean inequality)}\label{le2,1} %------
Let $x_1, x_2, \ldots, x_n$ be positive real numbers and $p$, $q$ real numbers such that $p>q$. Then,
$$\left(\frac{1}{n}\sum\limits_{i=1}^{n}x_i^p\right)^{\frac{1}{p}}\geq \left(\frac{1}{n}\sum\limits_{i=1}^{n}x_i^q\right)^{\frac{1}{q}}$$
with equality if and only if $x_1=x_2=\cdots=x_n$.
\end{lemma}

\begin{lemma}{\bf (H\"{o}lder inequality)}\label{le2,2} %------
Let $(a_1, a_2, \ldots, a_n)$ and $(b_1, b_2, \ldots, b_n)$ be two $n$-tuples of real
numbers and let $p$, $q$ be two positive real numbers such that $\frac{1}{p}+\frac{1}{q}=1$. Then
$$\left\lvert\sum\limits_{i=1}^{n}a_ib_i\right\rvert\leq \left(\sum\limits_{i=1}^{n}|a_i|^p\right)^{\frac{1}{p}} \left(\sum\limits_{i=1}^{n}|b_i|^q\right)^{\frac{1}{q}}$$
with equality if and only if $|a_i|^p=\lambda |b_i|^q$ for some real constant $\lambda$, $1\leq i\leq n$.
\end{lemma}

\begin{lemma}{\bf (\cite{SCE})}\label{le2,3} %------
If $p\geq 1$ is an integer and $0\leq x_1, x_2, \ldots, x_n\leq n-1$, then
$$\left(\sum\limits_{i=1}^{n}x_i^p\right)^{\frac{1}{p}} \leq (n-1)^{1-\frac{1}{p}}\sum\limits_{i=1}^{n}x_i^{\frac{1}{p}}.$$
\end{lemma}

\begin{lemma}{\bf (Minkowski inequality)}\label{le2,4} %------
Let $(a_1, a_2, \ldots, a_n)$ and $(b_1, b_2, \ldots, b_n)$ be two $n$-tuples of real
numbers. If $p\geq 1$. Then
$$\left(\sum\limits_{i=1}^{n}|a_i+b_i|^p\right)^{\frac{1}{p}}\leq \left(\sum\limits_{i=1}^{n}|a_i|^p\right)^{\frac{1}{p}}+\left(\sum\limits_{i=1}^{n}|b_i|^p\right)^{\frac{1}{p}}\eqno{(2.1)}$$
with equality if and only if $a_i=\lambda b_i$ for some real constant $\lambda$, $1\leq i\leq n$. For $0<p<1$, the inequality $(2.1)$ gets reversed.
\end{lemma}

\begin{lemma}{\bf (\cite{R})}\label{le2,5} %------
Let $x=(x_1, x_2, \ldots, x_k, \ldots)$ be a non-zero vector. Then for $p\geq 2$,
$$||x||_{p}\leq ||x||_2$$
with equality if and only if all but one of the $x_i$ are equal to $0$.
\end{lemma}

\begin{figure}[!hbpt]
\begin{center}
\includegraphics[scale=0.8]{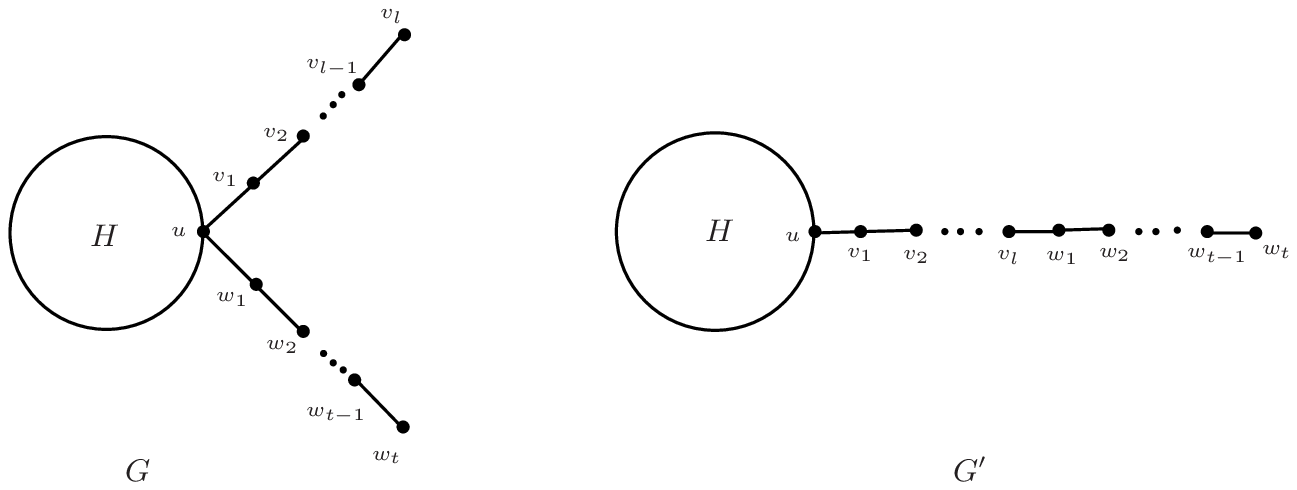}\\
Fig. 2.1~ Graphs $G$ and $G'$
\end{center}\label{fig1}
\end{figure}

\begin{lemma}\label{le2,6} %------
Let $G$ be a connected graph and $G'=G-uw_1+v_lw_1$ shown as Fig. 2.1. Let $d(u)\geq 3$, $u_i\in N(u)-\{v_1, w_1\}$, $i=1, 2,\ldots, d(u)-2$.

{\normalfont (i)} If $p>0$ and $d(u)>d(u_i)$, then $A_p(G)>A_p(G')$.

{\normalfont (ii)} If $p\geq 1$ and $d(u)\geq d(u_i)$, then $A_p(G)>A_p(G')$.

{\normalfont (iii)} If $p>0$ and $\Delta=d(u)=3$, then $A_p(G)>A_p(G')$.

\end{lemma}

\begin{proof} Let $s=d(u)-2$. Since $d(u)>d(u_i)$, $d(u)\geq 3$, $d(v_1)\leq 2$ and $d(w_1)\leq 2$, we have
$|d(u)-d(u_i)|^p-|d(u)-d(u_i)-1|^p>0$ and $|d(u)-d(v_1)|^p-|d(u)-d(v_1)-1|^p+1>0$. Then
\begin{eqnarray*}
A_p^p(G)-A_p^p(G') & = & \sum_{uv\in E(G)}|d(u)-d(v)|^p-\sum_{u'v'\in E(G')}|d(u')-d(v')|^p
\end{eqnarray*}
\begin{eqnarray*}
& = & \sum\limits_{i=1}^{s}(|d(u)-d(u_i)|^p-|d(u)-d(u_i)-1|^p)+|d(u)-d(w_1)|^p\\
& & +|d(u)-d(v_1)|^p-|d(u)-d(v_1)-1|^p+1\\
& > & 0.
\end{eqnarray*}
Thus we have $A_p(G)>A_p(G')$.

If $p\geq 1$ and $d(u)=d(u_i)$, then we have
\begin{eqnarray*}
A_p^p(G)-A_p^p(G') & = & \sum\limits_{i=1}^{s}(|d(u)-d(u_i)|^p-|d(u)-d(u_i)-1|^p)+|d(u)-d(w_1)|^p\\
& & +|d(u)-d(v_1)|^p-|d(u)-d(v_1)-1|^p+1\\
& = & -(d(u)-2)+|d(u)-d(w_1)|^p+|d(u)-d(v_1)|^p\\
& & -|d(u)-d(v_1)-1|^p+1\\
& > & 0.
\end{eqnarray*}
Thus we have $A_p(G)>A_p(G')$.

If $p>0$ and $\Delta=d(u)=3$, then we have
\begin{eqnarray*}
A_p^p(G)-A_p^p(G') & = & |d(u)-d(u_1)|^p-|d(u)-d(u_1)-1|^p+|d(u)-d(w_1)|^p\\
& & +|d(u)-d(v_1)|^p-|d(u)-d(v_1)-1|^p+1\\
& \geq  & -1+|3-d(w_1)|^p+|3-d(v_1)|^p\\
& & -|3-d(v_1)-1|^p+1\\
& = & |3-d(w_1)|^p+|3-d(v_1)|^p-|2-d(v_1)|^p\\
& > & 0.
\end{eqnarray*}
Thus we have $A_p(G)>A_p(G')$.

Combining the above arguments, we have the proof. $\qedsymbol$
\end{proof}

\section{\large  Some bounds for the general Albertson index}

\begin{theorem}\label{th3,1} %------
Let $G$ be a connected graph with $m$ edges. If $p>q$, then
$$A_p(G)\geq m^{\frac{1}{p}-\frac{1}{q}} A_{q}(G)$$
with equality if and only if $G$ is a regular graph (when $G$ is non-bipartite) or $G$ is a $(\Delta, \delta)$-semiregular bipartite graph (when $G$ is bipartite).
\end{theorem}

\begin{proof}By Lemma \ref{le2,1}, we have
$$\left(\frac{1}{m}\sum_{uv\in E(G)}|d(u)-d(v)|^p\right)^{\frac{1}{p}}\geq \left(\frac{1}{m}\sum_{uv\in E(G)}|d(u)-d(v)|^q\right)^{\frac{1}{q}},$$
that is,
$$\frac{1}{m^{\frac{1}{p}}}A_p(G)\geq \frac{1}{m^{\frac{1}{q}}}A_q(G),$$
that is,
$$A_p(G)\geq m^{\frac{1}{p}-\frac{1}{q}} A_{q}$$
with equality if and only if $|d(u)-d(v)|$ is a constant for every edge $uv$ in $G$, that is, $G$ is a regular graph (when $G$ is non-bipartite) or $G$ is a $(\Delta, \delta)$-semiregular bipartite graph (when $G$ is bipartite). $\qedsymbol$
\end{proof}

\begin{corollary}\label{cor3,1} %------
Let $G$ be a connected graph with $m$ edges. Then
$$Alb(G) \leq  \sqrt{m(F-2M_2)}$$
with equality if and only if $G$ is a regular graph (when $G$ is non-bipartite) or $G$ is a $(\Delta, \delta)$-semiregular bipartite graph (when $G$ is bipartite).
\end{corollary}

\begin{theorem}\label{th3,2} %------
Let $G$ be a connected graph. If $\frac{1}{p}+\frac{1}{q}=1$, then
$$ A_p(G)A_q(G)\geq F-2M_2$$
with equality if and only if $p=2$, or $G$ is a regular graph (when $G$ is non-bipartite) or $G$ is a $(\Delta, \delta)$-semiregular bipartite graph (when $G$ is bipartite).
\end{theorem}

\begin{proof} Let $a_i=b_i=d(u)-d(v)$ in Lemma \ref{le2,2}. Then
$$\sum\limits_{uv\in E(G)}(d(u)-d(v))^2\leq \left(\sum\limits_{uv\in E(G)}|d(u)-d(v)|^p\right)^{\frac{1}{p}} \left(\sum\limits_{uv\in E(G)}|d(u)-d(v)|^q\right)^{\frac{1}{q}},$$
that is,
$$F-2M_2\leq A_p(G)A_q(G)$$
with equality if and only if $p=2$, or $G$ is a regular graph (when $G$ is non-bipartite) or $G$ is a $(\Delta, \delta)$-semiregular bipartite graph (when $G$ is bipartite). $\qedsymbol$
\end{proof}

\begin{theorem}\label{th3,3} %------
Let $G$ be a connected graph with $m$ edges. If $p\geq 1$ is an integer, then
$$A_p(G) \leq (m-1)^{1-\frac{1}{p}}\left(A_{\frac{1}{p}}(G)\right)^{\frac{1}{p}}.$$
\end{theorem}

\begin{proof} Let $x_i=|d(u)-d(v)|$ in Lemma \ref{le2,3}. Then
$$\left(\sum\limits_{uv\in E(G)}|d(u)-d(v)|^p\right)^{\frac{1}{p}} \leq (m-1)^{1-\frac{1}{p}}\sum\limits_{uv\in E(G)}|d(u)-d(v)|^{\frac{1}{p}},$$
that is,
$$A_p(G) \leq (m-1)^{1-\frac{1}{p}}\left(A_{\frac{1}{p}}(G)\right)^{\frac{1}{p}}.~~\qedsymbol$$
\end{proof}

\begin{theorem}\label{th3,4} %------
Let $G$ be a connected graph with $m$ edges. If $p\geq 1$, then
$$A_p(G)\geq \left(m+pAlb(G)\right)^{\frac{1}{p}}-m^{\frac{1}{p}}.$$
If $0<p<1$, then
$$A_p(G)\leq \left(m+pAlb(G)\right)^{\frac{1}{p}}-m^{\frac{1}{p}}.$$
\end{theorem}

\begin{proof} Let $a_i=1$ and $b_i= |d(u)-d(v)|$ in Lemma \ref{le2,4}. Then
$$\left(\sum\limits_{uv\in E(G)}(1+|d(u)-d(v)|)^p\right)^{\frac{1}{p}} \leq A_p(G)+m^{\frac{1}{p}}$$
for $p\geq 1$. By Bernoulli inequality, we have
$$ A_p(G)+m^{\frac{1}{p}}\geq \left(\sum\limits_{uv\in E(G)}(1+p|d(u)-d(v)|)\right)^{\frac{1}{p}}=\left(m+pAlb(G)\right)^{\frac{1}{p}},$$
that is,
$$A_p(G)\geq \left(m+pAlb(G)\right)^{\frac{1}{p}}-m^{\frac{1}{p}}.$$
For $0<p<1$, by Lemma \ref{le2,4} and Bernoulli inequality, we have the proof. $\qedsymbol$
\end{proof}

\begin{theorem}\label{th3,5} %------
Let $G$ be a connected graph. If $p\geq 2$, then
$$A_p(G) \leq \sqrt{F-2M_2}\eqno{(3.1)}$$
with equality if and only if $d(u)-d(v)\neq 0$ for unique edge $uv$ and 0 for the other edges in $G$.
\end{theorem}

\begin{proof} By Lemma \ref{le2,5}, we have
$$A_p(G) \leq  A_2(G)=\left(\sum_{uv\in E(G)}|d(u)-d(v)|^2\right)^{\frac{1}{2}}=\sqrt{\sigma(G)}=\sqrt{F-2M_2}$$
with equality if and only if $d(u)-d(v)\neq 0$ for only one edge $uv$ and 0 for the other edges in $G$. $\qedsymbol$
\end{proof}

\begin{remark}
There exist many graphs such that the equality in $(3.1)$ holds, the following graphs are examples.
\begin{figure}[!hbpt]
\begin{center}
\includegraphics[scale=1.0]{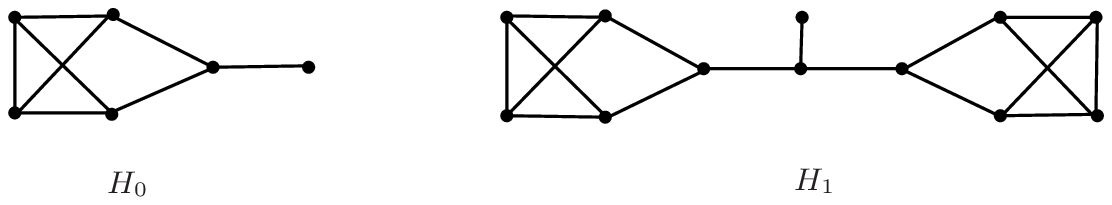}\\
Fig. 3.1~ Graphs $H_0$, $H_1$
\end{center}\label{fig2}
\end{figure}
\end{remark}

\begin{theorem}\label{th3,6} %------
Let $G$ be a connected graph with $n$ vertices. Then
$$A_p(G) \leq [ Z_p(L(G))]^{\frac{1}{p}}$$
with equality if and only if $G\cong K_{1,\ n-1}$.
\end{theorem}

\begin{proof} By definition of $A_p(G)$, we have
\begin{eqnarray*}
A_p(G) & = & \left(\sum_{uv\in E(G)}|d(u)-d(v)|^p\right)^{\frac{1}{p}}\\
& = & \left(\sum_{uv\in E(G)}|(d(u)-1)-(d(v)-1)|^p\right)^{\frac{1}{p}}\\
& \leq & \left[\sum_{uv\in E(G)}(d(u)+d(v)-2)^p\right]^{\frac{1}{p}}\\
& = & \left[\sum_{v\in V(L(G))}(d(v))^p\right]^{\frac{1}{p}}\\
& = &[ Z_p(L(G))]^{\frac{1}{p}}
\end{eqnarray*}
with equality if and only if $(d(u)-1)(d(v)-1)\leq 0$, that is, $d(v)=1$ for every edge $uv$ in $G$, that is, $G$ is a star $K_{1,\ n-1}$. $\qedsymbol$

\end{proof}

\begin{corollary}\label{cor3,2} %------
Let $G$ be a connected graph with $n$ vertices and $m$ edges. Then
$$Alb(G) \leq  Z_1(L(G))=Z_2(G)-2m \quad \text{and} \quad \sigma(G)\leq Z_2(L(G))$$
with equality if and only if $G\cong K_{1,\ n-1}$.
\end{corollary}

\begin{theorem}\label{th3,7} %------
Let $G$ be a connected graph with $n$ vertices. Then
$$A_p(u \vee G)= [ Z_p(\overline{G})+A_p^p(G)]^{\frac{1}{p}},$$
where $\overline{G}$ is the complement of $G$.
\end{theorem}

\begin{proof} By definition of $u \vee G$, we have
\begin{eqnarray*}
A_p(u \vee G) & = & \left(\sum_{uv\in E(u \vee G)}|d(u)-d(v)|^p\right)^{\frac{1}{p}}\\
& = & \left(\sum_{u\in V(G)}|n-d(u)-1|^p+\sum_{uv\in E(G)}|d(u)-d(v)|^p\right)^{\frac{1}{p}}\\
& = & \left[\sum_{v\in V(\overline{G})}d^p(v)+A_p^p(G)\right]^{\frac{1}{p}}\\
& = &[ Z_p(\overline{G})+A_p^p(G)]^{\frac{1}{p}}. \quad \qedsymbol
\end{eqnarray*}
\end{proof}

\begin{corollary}\label{cor3,3} %------
Let $G$ be a connected graph with $n$ vertices and $m$ edges. Then
\begin{eqnarray*}
Alb(u \vee G)-Alb(G) & = & n(n-1)-2m,\\
\sigma(u \vee G)-\sigma(G) +Z_2(G)& = & n(n-1)^2-4m(n-1).
\end{eqnarray*}
\end{corollary}

\begin{theorem}\label{th3,8} %------
Let $u$ be a pendant vertex of a connected graph $G$ with $n\geq 3$ vertices. If $G+P_t$ $(t\geq 1)$ is the graph by adding a new (pendant) path to $u$, then

{\normalfont (i)} $A_p(G+P_t)>A_p(G)$ for $0<p<1$.

{\normalfont (ii)} $A_p(G+P_t)=A_p(G)$ for $p=1$.

{\normalfont (iii)} $A_p(G+P_t)<A_p(G)$ for $p>1$.

\end{theorem}

\begin{proof} Let $v$ be the unique neighbour of $u$ in G. Since $a^p+b^p>(a+b)^p$ for $a>0$, $b>0$ and $0<p<1$, we have
\begin{eqnarray*}
A_p^p(G+P_t) & = & \sum_{rs\in E(G+e)}|d(r)-d(s)|^p\\
& = & (d(v)-2)^p+(2-1)^p+\sum_{rs\in E(G), r, s \neq u}|d(r)-d(s)|^p\\
& > & (d(v)-1)^p+\sum_{rs\in E(G), r, s \neq u}|d(r)-d(s)|^p\\
& = & A_p^p(G)
\end{eqnarray*}
for $0<p<1$. Thus $A_p(G+P_t)>A_p(G)$.
By a similar reasoning as above, we have the proof of (ii) and (iii). $\qedsymbol$
\end{proof}

\begin{corollary}\label{cor3,4} %------
Let $u$ be a pendant vertex of a connected graph $G$ with $n\geq 3$ vertices. If $G+P_t$ $(t\geq 1)$ is the graph by adding a new (pendant) path to $u$, then
$$\sigma(G+P_t)<\sigma(G).$$
\end{corollary}

\section{\large  The general Albertson index of trees}

\begin{theorem}\label{th4,1} %------
Let $T_n\in \mathscr{T}_{n}$. Then
$$2^{\frac{1}{p}}\leq A_p(T_n)\leq (n-2)(n-1)^{\frac{1}{p}}.$$
The lower bound is attained if and only if $T_n\cong P_n$. The upper bound is
attained if and only if $T_n\cong K_{1,\, n-1}$.
\end{theorem}

\begin{proof} If $\Delta \geq 3$, then $T_n$ has at least three pendant vertices. Thus $A_p(T_n)> 3^{\frac{1}{p}}>2^{\frac{1}{p}}=A_p(P_n)$. In addition, $A_p(T_n)\leq (\Delta-1)(n-1)^{\frac{1}{p}}\leq (n-2)(n-1)^{\frac{1}{p}}=A_p(K_{1,\, n-1})$. $\qedsymbol$
\end{proof}

\begin{theorem}\label{th4,2} %------
Let $n\geq 10$ and $T_n\in \mathscr{T}_{n}-\{P_n\}$.

{\normalfont (i)} If $p>1$, then
$ A_p(T_n)\geq 6^{\frac{1}{p}}$
with equality if and only if $T_n\cong T^1$.

{\normalfont (ii)} If $p=1$, then
$ A_p(T_n)\geq 6$
with equality if and only if $T_n$ is a spider with $\Delta=3$.

{\normalfont (iii)} If $0<p<1$, then
$ A_p(T_n)\geq (2^{p+1}+2)^{\frac{1}{p}}$
with equality if and only if $T_n\cong T^3$.
\end{theorem}

\begin{proof} Let $d_1\geq d_2\geq d_3\geq \cdots \geq d_n$ be the degree sequence of $T_n$, and let $k$ be the number of non-pendant edges $uv$ with $d(u)\neq d(v)$. Then
$$A_p^p(T_n)=\sum\limits_{uv\in E(G)}|d(u)-d(v)|^p\geq d_1+(d_2-2)+(d_3-2)+k.$$

If $d_1\geq 7$, then $A_p^p(T_n)> 7$.

If $d_1\geq 6$ and $d_2\geq 3$, then $A_p^p(T_n)> 6+(3-2)=7$.

If $d_1\geq 6$ and $d_2=2$, then $A_p^p(T_n)> 6+1=7$.

If $d_1=d_2=5$, then $A_p^p(T_n)>5+(5-2)=8$.

If $d_1= 5$ and $d_2=4$, then $A_p^p(T_n)>5+(4-2)=7$.

If $d_1=5$ and $d_2=3$, then $A_p^p(T_n)> 5+(3-2)+1=7$.

If $d_1=5$ and $d_2=2$, then $A_p^p(T_n)=\left\{
\begin{array}{llll}
4^{p+1}+3^p+1,\\
3\cdot 4^p+2\cdot 3^p+2,\\
2\cdot 4^p+3^{p+1}+3,\\
4^p+4\cdot 3^p+4,\\
5\cdot 3^p+5.
\end{array}
\right.$
Thus $A_p^p(T_n)>6$.

If $d_1=d_2=4$, then $A_p^p(T_n)\geq 4+(4-2)+1>7$.

If $d_1=4$ and $d_2=d_3=d_4=3$, then $A_p^p(T_n)> 4+(3-2)+(3-2)+(3-2)=7$.

If $d_1=4$ and $d_2=d_3=3$, then $A_p^p(T_n)> 4+(3-2)+(3-2)+1=7$.

If $d_1=4$ and $d_2=3$, then $A_p^p(T_n)> 4+(3-2)+1=6$.

If $d_1=4$ and $d_2=2$, then $A_p^p(T_n)=\left\{
\begin{array}{llll}
3^{p+1}+2^p+1,\\
2\cdot 3^p+2^{p+1}+2,\\
3^p+3\cdot 2^{p}+3,\\
4\cdot 2^p+4.
\end{array}
\right.$

If $d_1=3$, we can applying Lemma \ref{le2,6} repeatedly to the vertices with degree three. Thus the minimum value of $T_n$ has four cases, shown as in Fig. 4.1. By direct computing, we have
$$A_p^p(T^1)=2^{p+1}+2, \quad A_p^p(T^2)=A_p^p(T'^2)=2^p+4, \quad A_p^p(T^3)=6.$$

By comparing the above cases, we have that $T_1$, $T^2$, $T'^2$ and $T^3$ are the candidates with minimum general Albertson index among $\mathscr{T}_{n}-\{P_n\}$. Further, we have
$A_p^p(T^1)>A_p^p(T^2)=A_p^p(T'^2)>A_p^p(T^3)$ for $p>1$, $A_p^p(T^1)=A_p^p(T^2)=A_p^p(T'^2)=A_p^p(T^3)$ for $p=1$,
and $A_p^p(T^1)<A_p^p(T^2)=A_p^p(T'^2)<A_p^p(T^3)$ for $0<p<1$. $\qedsymbol$

\begin{figure}[!hbpt]
\begin{center}
\includegraphics[scale=0.87]{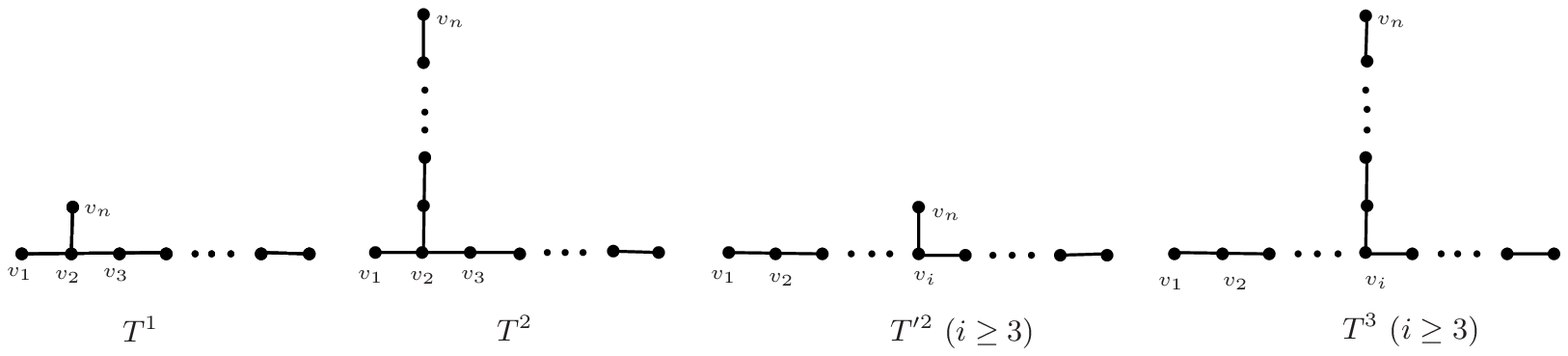}\\
Fig. 4.1~ Graphs $T^1$, $T^2$, $T'^2$, $T^3$
\end{center}\label{fig3}
\end{figure}
\end{proof}

\begin{theorem}\label{th4,3} %------
Let $T_n$ be a tree with $n$ vertices. If $p\geq1$, then
$$A_p(T_n)\geq (\Delta \varepsilon^{1-p}(v_{\Delta}))^{\frac{1}{p}}(\Delta-1),$$
where $v_{\Delta}$ is a vertex of maximum degree.
\end{theorem}

\begin{proof} Let $v_{\Delta}v_1v_2\ldots v_{l_1}$ be the path from the vertex $v_{\Delta}$ to pendant vertex $v_{l_1}$. Then $d(v_{\Delta}, v_{l_1})=l_1$. Let $f(x)=x^p$. Since $f(x)$ is a increasing and convex function for $x>0$ and $p\geq 1$, we have
\begin{align*}
 & \frac{1}{l_1}(f(|d(v_{\Delta})-d(v_1)|)+f(|d(v_{1})-d(v_2)|)+\cdots+f(|v_{l_1-1}-d(v_{l_1})|))\\
\geq {}& f\left(\frac{|d(v_{\Delta})-d(v_1)|+|d(v_{1})-d(v_2)|+\cdots+|v_{l_1-1}-d(v_{l_1})|}{l_1}\right)\\
\geq  {}& f\left(\frac{(d(v_{\Delta})-d(v_1)+d(v_{1})-d(v_2)+\cdots+v_{l_1-1}-d(v_{l_1}))}{l_1}\right)\\
=  {}& f\left(\frac{\Delta-1}{l_1}\right),
\end{align*}
that is,
$$|d(v_{\Delta})-d(v_1)|^p+\cdots+|v_{l_1-1}-d(v_{l_1})|^p\geq l_1^{1-p}(\Delta-1)^p.$$
Since $T_n$ has at least $\Delta$ pendant vertices, we have
$$A_p(T_n)\geq (l_1^{1-p}+l_2^{1-p}+\cdots+l_{\Delta}^{1-p})^{\frac{1}{p}}(\Delta-1),$$
where $l_1, l_2, \ldots, l_{\Delta}$ is the distance from maximum degree vertex $v_{\Delta}$ to pendant vertex $v_{l_i}$, $1\leq i \leq \Delta$.
Note that $\varepsilon(v_{\Delta})=\max_{v\in V(G)}d(v_{\Delta},v)\geq l_i$ for $1\leq i \leq \Delta$. Thus we have $A_p(T_n)\geq (\Delta \varepsilon^{1-p}(v_{\Delta}))^{\frac{1}{p}}(\Delta-1)$. $\qedsymbol$
\end{proof}

\begin{corollary}\label{cor4,1} %------
Let $T_n$ be a tree with $n$ vertices. Then
$$Alb(T_n) \geq \Delta (\Delta-1)$$
with equality if and only if $G$ is a spider.
\end{corollary}

\section{\large  The general Albertson index of generalized Bethe trees and Kragujevac trees}

In this section, we give the calculation formula of the general Albertson index of generalized Bethe trees and Kragujevac trees which are a wide range of
applications in the field of mathematics \cite{RJ, RT}, Cheminformatics \cite{HAG, SII, WWLJ}, statistical mechanics \cite{O}, etc.

A generalized Bethe tree \cite{RR} is a rooted tree in which vertices of the same level (height) have the same degree. We
usually use $B_k$ to denote the generalized Bethe tree with $k$ levels with the root at the level 1. More specifically, $B_{k,\, d}$ denotes a Bethe tree \cite{HL} of $k$ levels with root degree $d$, and the vertices between level $2$ and $k-1$ all have degree $d+1$. A regular dendrimer tree \cite{GYLC} $T_{k,\, d}$ is a special case of $B_k$, where the degrees of all internal vertices are $d$.

\begin{theorem}\label{th5,1} %------
Let $B_k$ be the generalized Bethe tree with the degree of each level is $d_1, d_2, \ldots, d_{k-1}, d_{k}=1$. Then
$$A_p(B_k)= d_1^{\frac{1}{p}}\left(|d_1-d_2|^p+\sum\limits_{i=2}^{k}|d_i-d_{i-1}|^p\prod\limits_{j=2}^{i}(d_j-1)\right)^{\frac{1}{p}}.$$
\end{theorem}

\begin{proof} By definition of the generalized Bethe tree, we have
\begin{eqnarray*}
A_p^p(B_k) & = & \sum_{uv\in E(G)}|d(u)-d(v)|^p\\
& = &  d_1[|d_1-d_2|^p+|d_2-d_3|^p(d_2-1)+|d_3-d_4|^p(d_3-1)(d_2-1)+\cdots\\
& & +|d_{k-1}-d_{k}|^p(d_{k-1}-1)\cdots(d_2-1)]\\
& = & d_1\left(|d_1-d_2|^p+\sum\limits_{i=2}^{k}|d_i-d_{i-1}|^p\prod\limits_{j=2}^{i}(d_j-1)\right).
\end{eqnarray*}
Thus we have the proof. $\qedsymbol$
\end{proof}

\begin{corollary}\label{cor5,1} %------
Let $B_{k,\, d}$ and $T_{k,\, d}$ be the Bethe tree and regular dendrimer tree, respectively. Then
$$A_p(B_{k,\, d})=(d+d^{p+k-1})^{\frac{1}{p}} \quad \text{and} \quad A_p(T_{k,\, d})=[d(d-1)^{p+k-2}]^{\frac{1}{p}}.$$
\end{corollary}
\begin{figure}[!hbpt]
\begin{center}
\includegraphics[scale=0.9]{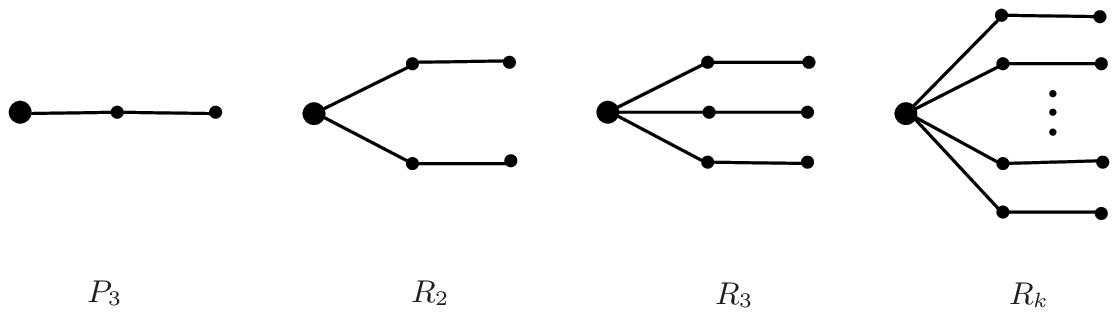}\\
Fig. 5.1~ Graphs $P_3$, $R_2$, $R_3$, $R_k$
\end{center}\label{fig4}
\end{figure}
Let $P_3$ be the 3-vertex tree, rooted at one of its terminal vertices, see Fig. 5.1. For $k=2, 3, \ldots$, construct the rooted tree $R_k$ by identifying the roots of $k$ copies of $P_3$. The vertex obtained by identifying the roots of $P_3$-trees is the root of $R_k$. Let $d\geq 2$ be an integer and $\gamma_1, \gamma_2, \ldots, \gamma_d$ be rooted trees, i.e., $\gamma_1, \gamma_2, \ldots, \gamma_d\in \{R_2, R_3, \ldots\}$. A Kragujevac tree $KT$ \cite{HAG} is a tree possessing a vertex of degree $d$, adjacent to the roots of $\gamma_1, \gamma_2, \ldots, \gamma_d$. This vertex is said to be the central vertex of $KT$, whereas $d$ is the degree of $KT$. The subgraphs $\gamma_1, \gamma_2, \ldots, \gamma_d$ are the branches of $KT$. Recall that some (or all) branches of $KT$ may be mutually isomorphic.

\begin{theorem}\label{th5,2} %------
Let $KT$ be a Kragujevac tree with $n$ vertices and $\gamma_i\cong R_{k_i}$, $i=1, 2, \ldots, d$. Then
$$A_p(KT)= \left[\frac{n-d-1}{2}+\sum\limits_{i=1}^{d}(k_i(k_i-1)^p+|k_i-d+1|^p)\right]^{\frac{1}{p}}.$$
\end{theorem}

\begin{proof} Since $1+\sum\limits_{i=1}^{d}(2k_i+1)=n$, by definition of the Kragujevac tree, we have
\begin{eqnarray*}
A_p^p(KT) & = & \sum_{uv\in E(G)}|d(u)-d(v)|^p\\
& = &  \sum\limits_{i=1}^{d}[k_i+k_i(k_i+1-2)^p+|d-(k_i+1)^p|]\\
& = & \frac{n-d-1}{2}+\sum\limits_{i=1}^{d}(k_i(k_i-1)^p+|k_i-d+1|^p).
\end{eqnarray*}
Thus we have the proof. $\qedsymbol$
\end{proof}

\begin{corollary}\label{cor5,2} %------
Let $KT$ be a Kragujevac tree with $n$ vertices and $\gamma_i\cong R_{k}$, $i=1, 2, \ldots, d$. Then
$$A_p(KT)= \left[\frac{n-d-1}{2}+dk(k-1)^p+d|k-d+1|^p\right]^{\frac{1}{p}}.$$
\end{corollary}

\section{\large Conclusion}

In this paper, we propose the general Albertson irregularity index which extends classical Albertson irregularity index and $\sigma$-index.
The tight bounds of the general Albertson irregularity index are established. Additionally, the general Albertson irregularity index of trees are studied. In 2014, the total irregularity of a graph $G$, introduced by Abdo, Brandt and Dimitrov \cite{ABD}, is defined as $\text{irr}_t(G) =\sum_{\{u,v\}\subseteq V(G)}|d(u)-d(v)|$. For measuring the non-self-centrality of a graph, the non-self-centrality number of $G$ was introduced in \cite{XDM} as $N(G) =\sum_{\{u,v\}\subseteq V(G)}|\varepsilon(u)-\varepsilon(v)|$. Based on these, we can propose the general total irregularity and the general non-self-centrality number of a graph $G$ as follows:
$$\text{irr}_{p}(G) =\left(\sum_{\{u,v\}\subseteq V(G)}|d(u)-d(v)|^p\right)^{\frac{1}{p}} \quad \text{and} \quad N_{p}(G) =\left(\sum_{\{u,v\}\subseteq V(G)}|\varepsilon(u)-\varepsilon(v)|^p\right)^{\frac{1}{p}},$$
where the summation goes over all the unordered pairs of vertices in $G$. The research interaction among $A_p(G)$, $\text{irr}_{p}(G)$ and $N_{p}(G)$ will be carried out in the near future.

\vskip 2mm

\small {

}

\end{document}